\documentclass[12pt]{amsart}
 \usepackage[latin1]{inputenc}
 \usepackage[dvips]{graphicx}
 \usepackage{wrapfig}
 \usepackage{amsmath}
 \usepackage{amsthm}
 \usepackage{amsfonts}
 \usepackage{amssymb}
 \usepackage{layout}
 \usepackage{verbatim}
 \usepackage{alltt}
\usepackage{yfonts}
\usepackage[T1]{fontenc}

\usepackage[all]{xy}
\usepackage{setspace}

\newtheorem*{thma}{Theorem A}
\newtheorem*{thmb}{Theorem B}
\newtheorem*{thmc}{Theorem C}

\newcommand{\LL}{\Lambda}
\newcommand{\QQ}{\mathbb{Q}}
\newcommand{\FF}{\mathcal{F}}
\newcommand{\GG}{\mathcal{G}}
\newcommand{\lra}{\longrightarrow}
\newcommand{\ZZ}{\mathbb{Z}}
\newcommand{\PP}{\mathcal{P}}
\newcommand{\NN}{\mathcal{N}}
\newcommand{\ra}{\rightarrow}

\newcommand{\be}{\begin{equation}}
\newcommand{\ee}{\end{equation}}

\newcommand{\XX}{\chi}
\newcommand{\kk}{\mathcal{K}}

\newcommand{\mm}{\hbox{\frakfamily m}}

\newcommand{\FFc}{\mathcal{F}_{\textup{\lowercase{can}}}}

\numberwithin{equation}{section}
\newtheorem{thm}{Theorem}[section]
\newtheorem{lemma}[thm]{Lemma}
\newenvironment{define}{\par\medskip\noindent\refstepcounter{thm}
\bgroup{\hspace*{-0.15 cm}\bf{Definition}
\thethm.}\bgroup}{\egroup \egroup\par\medskip}\newtheorem{prop}[thm]{Proposition}
\newtheorem{cor}[thm]{Corollary}
\newenvironment{rem}{\par\medskip\noindent\refstepcounter{thm}
\bgroup{\hspace*{-0.15 cm}\bf{Remark} \thethm.}\bgroup}{\egroup
\egroup\par\medskip} \parskip 2pt

\newenvironment{example}{\par\medskip\noindent\refstepcounter{thm}
\bgroup{\hspace*{-0.15 cm}\bf{Example}
\thethm.}\bgroup}{\egroup \egroup\par\medskip}

\begin{document}
\title{{T}\lowercase{amagawa  defect of} {E}\lowercase{uler Systems}}

\author{K\^az\i m B\"uy\"ukboduk}

\address{Kazim Buyukboduk \hfill\break\indent
IH\'ES, Le Bois-Marie, 35,
\hfill\break\indent Route de Chartres \hfill\break\indent F-91440 Bures-sur-Yvette
\hfill\break\indent FRANCE} \email{{\tt kazim@ihes.fr }\hfill\break\indent {\it
Web page:} {\tt http://math.stanford.edu/$\sim$kazim}}


\begin{abstract}

As remarked in~\cite{mr02} Proposition 6.2.6 and~\cite{kbbonline} Remark 3.25 one does not expect the Kolyvagin system obtained from an Euler system for a $p$-adic Galois representation $T$ to be \emph{primitive} (in the sense of~\cite{mr02} Definition 4.5.5) if $p$ divides a Tamagawa number at a prime $\ell\neq p$; thus fails to compute the correct size of the relevant Selmer module. In this paper we obtain a lower bound for the size of the cokernel of the Euler system to Kolyvagin system map (see Theorem 3.2.4 of~\cite{mr02} for a definition) in terms of the Tamagawa numbers of $T$, refining~\cite{mr02} Propostion 6.2.6. We show how this partially accounts for the missing Tamagawa factors in Kato's calculations with his Euler system in~\cite{kato}.

\end{abstract}

\maketitle

\section{Introduction}

Let $p>2$ be a rational prime and let $\mathcal{O}$ be the ring of integers of a finite extension of $\QQ_p$. Denote the maximal ideal of $\mathcal{O}$ by $\mm$ and fix a generator $\pi$ of $\mm$. Let $T$ be a free $\mathcal{O}$-module of finite rank, on which the absolute Galois group $G_\QQ:=\textup{Gal}(\overline{\QQ}/\QQ)$ acts continuously, and the action of $G_\QQ$ on $T$ is unramified outside a finite number of places. For such a $T$, the notion of an Euler system (which is originally due to Kolyvagin~\cite{koly}) has been generalized in Rubin~\cite{r00}, Kato~\cite{kato2} and Perrin-Riou~\cite{pr5} to prove upper bounds for the Selmer group attached to the Cartier dual $T^*$ of the Galois representation $T$. 
 
Starting from an Euler system, Kolyvagin uses his descent argument to obtain what he calls \emph{derivative classes}. These derivative classes are used to produce bounds for the dual Selmer group. \cite{mr02} starts exactly with these classes, and they observe that the derivative classes enjoy stronger local conditions than has been previously utilized. Classes with these stronger local conditions (and with the same interrelations that the derivative classes ought to satisfy) are called \emph{Kolyvagin systems}. We refer the reader to~\cite{mr02} \S3 for a detailed description. Since Kolyvagin systems are modeled after the derivative classes, they have exactly the same applications, namely they give upper bounds for the dual Selmer group. In fact Mazur and Rubin exploits the extra rigidity gained by their observation to prove, in many cases of interest, that the Kolyvagin system bound on the dual Selmer group is strict and in fact one could completely determine the structure of the dual Selmer group in terms of a Kolyvagin system (if the Kolyvagin system we use is \emph{primitive} in the sense of Definition~\ref{primitive} below); see \cite{mr02} Theorems 4.5.6, 4.5.9 and 5.2.14.

The discussion above already portrays Kolyvagin systems as more fundamental objects than Euler systems. In fact, it is also possible to prove that Kolyvagin systems exist in many cases, however, it is impossible to write these down explicitly in full generality. The only cases where the bound provided by a Kolyvagin system can be made explicit are the cases where the Kolyvagin system used comes from an Euler system, via Kolyvagin's descent. This map from the collection of Euler systems to the collection of Kolyvagin systems will be referred to as the \emph{Euler system to Kolyvagin system map}; see Theorem~\ref{esksmap} below for a slightly more detailed description of this map. 

One important feature of the bounds provided by a Kolyvagin system obtained from one of the Euler systems known to date is that they are closely related to the special values of $L$-functions. Such bounds thus provide evidence for the Bloch-Kato conjectures~\cite{blochkato}, which predict the orders of these Selmer groups in terms of the special values of a relevant $L$-function. 

A natural question to ask is when these bounds given by an Euler system (or, equivalently, by the Kolyvagin system obtained from it) are sharp. In view of the results of~\cite{mr02}, this is equivalent to (under certain technical assumptions) asking when the Kolyvagin system obtained from the Euler system we started with is primitive. For example, consider Kato's Euler system~\cite{kato}, which is an Euler system for $T=T_p(E)$, the $p$-adic  Tate module of an elliptic curve $E_{/\QQ}$. As explained in~\cite{mr02} Remark 6.2.5, Kato's Euler system does not give rise to a primitive Kolyvagin system if $p$ divides one of the Tamagawa numbers of $E$. In fact they prove that the Euler system to Kolyvagin system map in this setting has non-trivial cokernel if $p$ divides a Tamagawa number, and as a result it is impossible to obtain a primitive Kolyvagin system from an Euler system in this case. We call this phenomenon the \emph{Tamagawa defect of Euler systems}.

However, the arguments of~\cite{mr02} (particularly Proposition 6.2.6 of \emph{loc. cit.}) is not sufficient to obtain an improved lower bound on the size of this cokernel in terms the Tamagawa factors. This is what we do in this paper:

\begin{thma}
Suppose $\pi^n$ divides a Tamagawa number of the Galois representation $T$. Under suitable hypotheses on $T$ the image of the Euler system to Kolyvagin system map (of Theorem~\ref{esksmap}) is contained in $\mm^n\textup{\textbf{KS}}(T)$, where $\textup{\textbf{KS}}(T)$ denotes the $\mathcal{O}$-module of Kolyvagin systems for $T$.
\end{thma}
See Theorem~\ref{ESKS cokernel} below for details.

In particular, the hypotheses of Theorem A hold when $T=T_p(E)$, i.e. when $T$ is the $p$-adic Tate module of an elliptic curve $E_{/\QQ}$ with conductor $N$. Let $\kappa^{\textup{Kato}}$ denote the Kolyvagin system obtained from Kato's Euler system, as in~\cite{mr02} \S6.2.  Let $c_\ell$ be the Tamagawa number of $E$ at $\ell$, and suppose $p^n|c_\ell$.

\begin{thmb}
$\kappa^{\textup{Kato}} \in p^n\textup{\textbf{KS}}(T).$ 
\end{thmb}

As a corollary, this shows that the bound obtained using $\kappa^{\textup{Kato}}$ (see \cite{mr02}~Theorem 6.2.4, for example) can be improved as follows:
\begin{thmc} Let $\textup{TS}_E$ be the Tate-Shafarevich group of $E$, $L_N(E,s)$ the "non-primitive" Hasse-Weil $L$-function associated to $E$ with Euler factors at the primes dividing $N$ removed and $\Omega_E$ the fundamental real period of $E$. Then
$$\textup{length}(\textup{TS}_E[p^\infty]) \leq  \textup{ord}_p\left(\frac{L_N(E,1)}{c_\ell\cdot\Omega_E}\right).$$
\end{thmc}
See also Theorem~\ref{selmersize} below.

As one may notice, the "improvement" we give above includes only one Tamagawa factor. A further improvement which shall include all Tamagawa factors unfortunately escapes our method. We discuss this matter further in \S~\ref{concluding}. We also elaborate on the hypotheses of Theorem~\ref{ESKS cokernel} to produce other interesting occurrences of the Tamagawa defect of Euler systems for representations other than the Tate module of an elliptic curve. 

Our results are somewhat related to that of~\cite{jetchev},  however they are more general in the sense that our Theorem~\ref{ESKS cokernel} gives a conceptual explanation for the Tamagawa defect for many Galois representations; yet the setting for which Theorem~\ref{ESKS cokernel} applies is disjoint from that of Jetchev's as far as the Kolyvagin system machinery is concerned.
\section*{Acknowledgements}

We would like to express our gratitude to Karl Rubin, who encouraged us to compute the cokernel of the specialization map on the module of $\LL$-adic systems. This paper is essentially a response to his suggestions. We also thank Olivier Fouquet and Jan Nekov\'a\v{r} for their feedback; Ofer Gabber and Tetsushi Ito for valuable conversations about the component groups of abelian varieties. Finally, we would like to thank IH\'ES, where this paper was written up, for their warm hospitality.
  
\section{Technical Results}
\subsection{Basic Definitions}
\subsubsection{Notation}
Fix once and for all an odd rational prime $p$. Let $R$ be a local principal ideal ideal ring with finite residue field of characteristic $p$, $\mm$ be its maximal ideal and $\mathbf{k}=R/\mm$ be its residue field. We fix a generator $\pi$ of $\mm$. For the main applications of our technical results $R$ will the ring of integers of a finite extension $\QQ_p$, in that case we write $F$ for its field of fractions.

For any field $K$ (local or global) $\overline{K}$ will be a fixed separable algebraic closure of $K$ and $G_K$ will denote Galois group $\textup{Gal}(\overline{K}/K)$. For every rational prime $\ell$ we fix an embedding $G_{\QQ_\ell} \hookrightarrow G_\QQ$. This fixes a decomposition group of $\ell$, and we write $I_\ell$ for the inertia subgroup inside of this fixed decomposition group. 

Let $T$ be an $R$-module endowed with a continuous $G_{\QQ}$-action, which is free of finite rank over $R$. We will assume that $T$ is unramified outside finitely many primes. If $R$ is the ring of integers of a finite extension $\QQ_p$, we write $V$ for $T\otimes_R F$ and $W$ for $T \otimes_R F/R=V/T$. By $H^{*}(K,X):=H^{*}(G_K,X)$ we mean the group cohomology of $G_K$ computed with respect to continuous cochains with values in $X$ for $X=T,V,W$ or their subquotients. 

If a group $H$ acts on a set $X$, then the subset of elements of $X$ fixed (pointwise) by $H$ is denoted by $X^H$.

If $M$ is an $R$-module and $I$ is an ideal of $R$, then $M[I]$ will denote the submodule of $M$ killed by $I$. If $M$ is a $G_{\QQ}$-module, $\QQ(M)$ will be defined as the fixed field in $\overline{\QQ}$ of the kernel of the map $G_{\QQ} \ra \textup{Aut}(M)$.

\subsubsection{Local Cohomology Groups and Local Conditions}
\label{subsub:Local cohom}
Much of the definitions and results we record in \S\ref{subsub:Local cohom} can be found in Chapter 1 of~\cite{mr02}.

Throughout this section let $K$ denote a non-archimedean local field and $\overline{K}$ a fixed separable algebraic closure of $K$. $\mathcal{O}$ will be the ring of integers in $K$, $\mathbb{F}$ its residue field, and $K^{\textup{unr}} \subset \overline{K}$ the maximal unramified subfield of $\overline{K}$. Let $\mathcal{I}$ be the inertia subgroup $\textup{Gal}(\overline{K}/K^{\textup{unr}})$, and $G_{\mathbb{F}}=\textup{Gal}(K^{\textup{unr}}/K)$. \\\\
\ref{subsub:Local cohom}.1.  \textit{Galois Cohomology of Local Fields.}
There is an exact sequence of profinite groups $$\{1\}\lra \mathcal{I} \lra G_K \lra G_{\mathbb{F}} \lra \{1\}$$ Further, since the cohomological dimension of $G_{\mathbb{F}}\cong\hat{\ZZ}$ is one it follows that  $H^2(G_{\mathbb{F}}, T^{\mathcal{I}})$ vanishes.  Thus the Hochschild-Serre spectral sequence gives rise to the following exact sequence:
$$0\lra H^1(G_{\mathbb{F}},T^{\mathcal{I}}) \lra H^1(K,T) \lra H^1(\mathcal{I}, T)^{G_{\mathbb{F}}} \lra 0$$\\
\ref{subsub:Local cohom}.2.  \textit{Local Conditions.}
\begin{define}
\label{def:local condition}
A \emph{local condition} $\FF$ on $T$ (at $\ell$ if $K=\QQ_\ell$) is a choice of an $R$-submodule $H^1_{\FF}(K,T)$ of $H^1(K,T)$.
\end{define}
Suppose $T$ is an $R$-module with a continuous $G_K$-action, and $\FF$ is a local condition on $T$. If $T^{\prime}$ is a submodule of $T$ (\emph{resp.} $T^{\prime \prime}$ is a quotient module), then $\FF$ induces local conditions (which we still denote by $\FF$) on $T^{\prime}$ (\emph{resp.} on $T^{\prime \prime}$), by taking $H^1_{\FF}(K,T^{\prime})$ (\emph{resp.} $H^1_{\FF}(K,T^{\prime \prime})$) to be the inverse image (\emph{resp.} the image) of $H^1_{\FF}(K,T)$ under the natural maps induced by $$T^{\prime} \hookrightarrow T, \,\,\, \,\,\,\,\,\,\, T \twoheadrightarrow T^{\prime \prime}.$$
\begin{define}
\label{def:propagation}
\emph{Propagation} of a local condition $\FF$ on $T$ to a submodule $T^{\prime}$ (and a quotient $T^{\prime \prime}$ of $T$ is the local condition $\FF$ on $T^{\prime}$ (and on $T^{\prime \prime}$) obtained following the above procedure.
\end{define}
For example, if $I$ is an ideal of $R$, then a local condition on $T$ induces local conditions on $T/IT$ and $T[I]$, by \emph{propagation}.

Let $\textup{Quot}_R(T)$ be the category $R[[G_K]]$-modules whose objects are quotients $T/IT$ for all ideals $I$ of $R$, and where the morphisms from $T/IT$ to $T/JT$ are all scalar multiplications $r$ such that $rI \subset J$.
\begin{define}
\label{cartesian}
A local condition $\FF$ is \emph{cartesian} on $\textup{Quot}_R(T)$ (or on a subcategory of $\textup{Quot}_R(T)$) if for any injective $R[[G_K]]$-module homomorphism $\phi: T_1 \ra T_2$ the local condition $\FF$ on $T_1$ is the same as the local condition obtained by propagating $\FF$ from $T_2$ to $T_1$.
\end{define}
\ref{subsub:Local cohom}.3.  \textit{Examples of Local Conditions}.
\label{examples-local}
We review several choices for local conditions which will appear quite frequently.
\begin{define}
Suppose $L$ is an extension of $K$ in $\overline{K}$, and define \begin{align*}H^1_L(K,T):=H^1(\textup{Gal}(L/K),&T^{G_L})=\\ &\ker\left\{H^1(K,T) \ra H^1(L,T)\right\} \subset H^1(K,T) \end{align*}
\end{define}

Thus every choice of an algebraic extension $L/K$ gives a choice of a local condition. We note that $H^1_L(K,T)$ is functorial in $T$. The \emph{unramified} condition  frequently appears in this paper and is obtained by taking $L=K^{\textup{unr}}$. Namely $$H^1_{\textup{unr}}(K,T):=H^1_{K^{\textup{unr}}}(K,T)=H^1(G_{\mathbb{F}},T).$$ When $T$ is unramified (i.e. $\mathcal{I}$ acts trivially on $T$), we will also call this the \emph{finite} condition and write $H^1_{\textup{f}}(K,T)=H^1_{\textup{unr}}(K,T)$.

In general, if $\textup{char}(\mathbb{F})\neq p$ and $R$ is the ring of integers of a finite extension $\QQ_p$, the finite condition at $K$ is given by $$H^1_{\textup{f}}(K,T)=\ker\left\{H^1(K,T) \lra H^1(K^{\textup{unr}},V)\right\}.$$ See~\cite{r00} \S3.1 for a more detailed discussion on the finite and unramified local conditions. 
\\\\
\ref{subsub:Local cohom}.4.\textit{Dual Local Conditions}.
\begin{define}
\label{cartier dual}
Define the \emph{Cartier dual} of $T$ to be the $R[[G_K]]$-module $$T^*:=\textup{Hom}(T,\mu_{p^{\infty}})$$ where $\mu_{p^{\infty}}$ stands for the $p$-power roots of unity inside $\overline{\QQ_p}$.
\end{define}
There is the perfect local Tate pairing $$<\,,\,>\,:H^1(K,T) \times H^1(K,T^*) \lra H^2(K,\mu_{p^{\infty}}) \stackrel{\sim}{\lra}\QQ_p/\ZZ_p$$
\begin{define}
\label{dual local condition}
The \emph{dual local condition} $\FF^*$ on $T^*$ of a local  condition $\FF$ on $T$ is defined so that $H^1_{\FF^*}(K,T^*)$ is the orthogonal complement of $H^1_{\FF}(K,T)$ under the local Tate pairing $<\,,\,>.$
\end{define}

\begin{prop}\textup{(Proposition 1.3.2 in~\cite{mr02})}
Suppose that the residue characteristic of the local field $K$ is different from $p$. Then $H^1_{\textup{f}}(K,T)$ and $H^1_{\textup{finite}}(K,T^*)$ are orthogonal complements under the local Tate pairing $<\,,\,>$.
\end{prop}
\subsubsection{Selmer structures and Selmer groups}
Definitions and results we record in this section can be found in Chapter 2 of~\cite{mr02}.

For the rest of this paper, unless otherwise is stated, $T$ will be a free $R$-module endowed with a continuous action of $G_{\QQ}$, which is unramified outside a finite set of rational primes. Below notation will also be in effect till the end.

Let $\overline{\QQ} \subset \mathbb{C}$ be the algebraic closure of $\QQ$ in $\mathbb{C}$, and for each rational prime $\ell$ we fix an algebraic closure $\overline{\QQ_{\ell}}$ of $\QQ_{\ell}$ containing $\overline{\QQ}$. We will ignore the infinite place of $\QQ$ systematically since we assumed $p>2$. Occasionally we will denote $G_{\QQ_{\ell}}=\textup{Gal}(\overline{\QQ_{\ell}}/\QQ_{\ell})$ by $\mathcal{D}_{\ell}$, whenever we would like to identify this group by a closed subgroup of $G_{\QQ}=\textup{Gal}(\overline{\QQ}/\QQ)$; namely with a particular decomposition group at $\ell$ in $G_{\QQ}$. We further define $\mathcal{I}_{\ell} \subset \mathcal{D}_{\ell}$ to be the inertia group and $\textup{Fr}_{\ell} \in \mathcal{D}_{\ell}/\mathcal{I}_{\ell}$ to be the arithmetic Frobenius element  at $\ell$. 
\begin{define}
\label{selmer structure}
A \emph{Selmer structure} $\FF$ on $T$ is a collection of the following data:
\begin{itemize}
\item a finite set $\Sigma(\FF)$ of places of $\QQ$, including $\infty, p$, and all primes where $T$ is ramified.
\item for every $\ell \in \Sigma(\FF)$ a local condition (in the sense of Definition~\ref{def:local condition}) on $T$ (which we view now as a $R[[\mathcal{D}_{\ell}]]$-module), i.e., a choice of $R$-submodule $$H^1_{\FF}(\QQ_{\ell},T) \subset H^1(\QQ_{\ell},T).$$ 
 \end{itemize}
If $\ell \notin \Sigma(\FF)$ we will also write $H^1_{\FF}(\QQ_{\ell},T)=H^1_{\textup{f}}(\QQ_{\ell},T)$.
\end{define}

\begin{define}
\label{slemer group}
If $\FF$ is a Selmer structure, we define the \emph{Selmer module} $H^1_{\FF}(\QQ,T)$ to be the kernel of the sum of the restriction maps $$H^1(\textup{Gal}(\QQ_{\Sigma(\FF)}/\QQ),T) \lra \bigoplus_{\ell \in \Sigma(\FF)}H^1(\QQ_{\ell},T)/H^1_{\FF}(\QQ_{\ell},T)$$ where $\QQ_{\Sigma(\FF)}$ is the maximal extension of $\QQ$ which is unramified outside $\Sigma(\FF)$.
\end{define}

\begin{example}
\label{canonical selmer}
Suppose $R$ is the ring of integers of a finite extension $\QQ_p$. The \emph{canonical Selmer structure} $\FF_{\textup{can}}$ on $T$ is given by 
\begin{itemize}
\item $\Sigma(\FFc)=\{\ell: T \hbox{ is ramified at } \ell\} \cup \{p,\infty\}$,
\item if $\ell \in \Sigma(\FFc)$ and $\ell \neq p, \infty$ then $H^1_{\FF_{\textup{can}}}(\QQ_{\ell},T)=H^1_{\textup{f}}(\QQ_\ell,T)$, 
\item $H^1_{\FF_{\textup{can}}}(\QQ_{p},T)=H^1(\QQ_{p},T)$.
\end{itemize}
Note that we may safely ignore the infinite place since $p>2$, therefore $H^1(\mathbb{R},T)=0$. 

If $I$ is an ideal of $R$ we define the canonical Selmer structure on $T/IT$ (which we still denote by $\FF_{\textup{can}}$) to be the Selmer structure obtained from $\FF_{\textup{can}}$ on $T$ by \emph{propagation} of local conditions.
\end{example}

\begin{define}
\label{def:selmer triple}
A \emph{Selmer triple} is a triple $(T,\FF,\PP)$ where $T$ is an $R[[G_{\QQ}]]$-module which is free of finite rank over $R$, unramified outside finitely many primes; $\FF$ is a Selmer structure on $T$ and $\PP$ is a set of rational primes, disjoint from $\Sigma(\FF)$.
\end{define}

\subsection{Hypotheses}
\label{hypo}
 In this section we record the hypotheses which were utilized by Mazur, Rubin and Howard to prove their main theorems on Kolyvagin systems in~\cite{mr02}. For a discussion  of these hypotheses see \S3.5 of~\cite{mr02}. 

\begin{description}
\item[\textbf{H1}] $T/\mm T$ is an absolutely irreducible $\mathbf{k}[G_{\QQ}]$-representation.
\item[\textbf{H2}] There is a $\tau \in G_{\QQ}$ such that $\tau=1$ on $\mu_{p^{\infty}}$ and $T/(\tau-1)T$ is free of rank one over $R$.
\item[\textbf{H3}] $H^{1}(\QQ (T  , \mu_{p^{\infty}})/\QQ  , T/\mm T)$ = $H^{1}(\QQ (T, \mu_{p^{\infty}})/\QQ, T ^{*}[\mm])=0$. 
\item[\textbf{H4}] Either

\textbf{H4a} $\textup{Hom}_{\mathbf{k}[[\textup{Gal}(\overline{\QQ}/\QQ)]]}(T/\mm T, T ^{*}[\mm])=0$, or

\textbf{H4b} $p>4$.
\item[\textbf{H5}] $\PP_t \subset \PP \subset \PP_1$ for some $t \in \ZZ^+$, where $\PP_k$ is as in \cite{mr02}~Definition 3.1.6.
\item[\textbf{H6}] For every $\ell \in \Sigma(\FF)$, the local condition $\FF$ at $\ell$ is cartesian (in the sense of Definition~\ref{cartesian}) on the category $\textup{Quot}_R(T)$ of quotients of $T$.
 \end{description}
\begin{rem}
\begin{enumerate}
\item Suppose $R$ is the ring of integers of a finite extension $\QQ_p$. Then $\FFc$ satisfies $\textbf{H6}$ by \cite{mr02}~Lemma 3.7.1.
\item Suppose that $E_{/\QQ}$ is an elliptic curve defined over $\QQ$ which does not have complex multiplication, and let $T=T_p(E)$ be its $p$-adic Tate module (which is a representation of $G_\QQ$ which is free of rank 2 over $R=\ZZ_p$). It is verified in~\cite{ru98} that $T_p(E)$ satisfies the hypotheses \textbf{H1-H4}, and the choice the set of primes $\PP$ which satisfies $\textbf{H5}$ has been explained (see also~\cite{scholl}). Thus the hypotheses above hold for the Selmer triple $(T,\FFc,\PP)$.
\end{enumerate} 
\end{rem}
\subsection{Theorems of Howard, Mazur and Rubin}
\label{thmsMR}
Suppose $(T,\FF,\PP)$ is a Selmer triple (in the sense of Definition~\ref{def:selmer triple}).  Let $\textup{\textbf{KS}}(T)=\textup{\textbf{KS}}(T,\FF,\PP)$ denote the $R$-module of Kolyvagin systems for  the Selmer triple $(T,\FF,\PP)$ defined as in~\cite{mr02} Definition 3.1.3. We also let $\overline{\textup{\textbf{KS}}}(T,\FF,\PP)$ be the \emph{generalized} module of Kolyvagin systems, see \cite{mr02}  Definition 3.1.6 for a definition. Under the hypotheses set in \S\ref{hypo} above Howard, Mazur and Rubin show that the structure of the modules $\textup{\textbf{KS}}(T,\FF,\PP)$ and $\overline{\textup{\textbf{KS}}}(T,\FF.\PP)$ is determined by an invariant $\XX(T)=\XX(T,\FF)$ which they call the \emph{core Selmer rank}, see~\cite{mr02} Definitions 4.1.11 and 5.2.4. In~\S\ref{thmsMR} we give a survey of their relevant results. 

\subsubsection{Core Selmer rank and the module of Kolyvagin systems } Recall the definition of the canonical Selmer structure $\FFc$. Theorem below (which is~\cite{mr02} Theorem 5.2.15) enables us to calculate the core Selmer rank $\XX(T,\FFc)$ of the  canonical Selmer structure $T$:
\begin{thm}
Suppose $R$ is  a discrete valuation ring. Let $d^-=\textup{rank}_R(T^-)$, where $T^-$ is the $-1$-eigenspace for the action of some complex conjugation. Then
$$\XX(T,\FFc)=d^-+\textup{corank}_R(H^0(\QQ_p,T^*)).$$
\end{thm}

\begin{example}
\label{core-ell}
Suppose $E_{/\QQ}$ is an elliptic curve defined over $\QQ$ and let $T=T_p(E)$ be its $p$-adic Tate module. In this case $\chi(T,\FFc)=\textup{rank}_{\ZZ_p}T^-=1$.
\end{example}
Fix a Selmer triple $(T,\FF,\PP)$  until the end of \S\ref{thmsMR}, for which \textbf{H1-H6} hold.
\begin{thm}\label{thmpa}\textup{(\cite{mr02} Corollaries 4.5.1 and 4.5.2)}
Suppose $R$ is a principal artinian ring of length $k$.
\begin{itemize}
\item[(i)] If $\XX(T)=0$ then $\textup{\textbf{KS}}(T)=0$.
\item[(ii)] If $\XX(T)\geq2$ then for every positive integer $d$, $\textup{\textbf{KS}}(T)$ contains a free $R$-module of rank $d$.
\item[(iii)] Suppose $\XX(T)=1$. Then, \begin{enumerate}
\item $\textup{\textbf{KS}}(T)$ is a free $R$-module of rank one.
\item If $j\leq k$ then the projection $T\ra T/\mm^jT$ induces a surjective map $\textup{\textbf{KS}}(T) \ra \textup{\textbf{KS}}(T/\mm^jT)$.
\end{enumerate}
\end{itemize}
\end{thm}

Building on Theorem~\ref{thmpa}, the following result is proved in~\cite{mr02} Proposition 5.2.9 and Theorem 5.2.10:

\begin{thm}\label{thmdvr} Suppose $R$ is a discrete valuation ring. 
\begin{itemize}
\item[(i)] If $\chi(T)=0$ then $\textup{\textbf{KS}}(T)=0.$
\item[(ii)] Suppose $\chi(T)=1$. Then,
\begin{enumerate}
\item $\textup{\textbf{KS}}(T) \stackrel{\sim}{\lra}\varprojlim\textup{\textbf{KS}}(T/\mm^k T,\PP_k) \stackrel{\sim}{\lra} \overline{\textup{\textbf{KS}}}(T),$
\item $\textup{\textbf{KS}}(T)$ is a free $R$-module of rank one, generated by a $\kappa \in \textup{\textbf{KS}}(T)$ whose image in $\textup{\textbf{KS}}(T/\mm T)$ is nonzero.
\end{enumerate}
\end{itemize}
\end{thm}
\begin{define}\label{primitive} $\kappa \in \textup{\textbf{KS}}(T)$ is called \emph{primitive} if the image of $\kappa$ in $\textup{\textbf{KS}}(T/\mm T)$ is nonzero.
\end{define}
\subsubsection{Euler systems and the descent map} Suppose for this section that $R$ is the ring of integers of a finite extension of $\QQ_p$.  Let $(T,\FFc,\PP)$ be a Selmer triple, and let $\kk$ be an abelian extension of $\QQ$ which contains the maximal abelian $p$-extension of $\QQ$ which is unramified outside $p$ and $\PP$. Following~\cite{mr02} Definition 3.2.2 we let $\textup{\textbf{ES}}(T)=\textup{\textbf{ES}}(T,\PP,\kk)$ denote the collection of Euler systems for $(T,\PP,\kk)$.

\begin{thm}\label{esksmap}\textup{(\cite{mr02} Theorem 3.2.4)} Suppose that $T/(\textup{Fr}_\ell-1)T$ is a cyclic $R$-module for every $\ell \in \PP$, and that $\textup{Fr}_\ell^{p^k}-1$ is injective on $T$ for every $\ell \in \PP$ and every $k\geq 0$. Then there is a canonical homomorphism $\textup{\textbf{ES}}(T)\ra \overline{\textup{\textbf{KS}}}(T,\FFc,\PP)$ such that if $\mathbf{c}\in \textup{\textbf{ES}}(T)$ maps to $\kappa \in \overline{\textup{\textbf{KS}}}(T,\FFc,\PP)$, then $\kappa_1=\mathbf{c}_\QQ$.

\end{thm}
\subsection{Comparison of Selmer structures and the Cartesian Condition}
\begin{lemma}
\label{Lemma371}
Suppose for the local condition $H^1_\FF(\QQ_\ell,T) \subset H^1(\QQ_\ell,T)$ the $R$-module $H^1(\QQ_\ell,T)/H^1_\FF(\QQ_\ell,T)$ is torsion-free. Then for every $n\in\ZZ^+$ the induced local condition on the quotients $\textup{Quot}_{R/\mm^n}(T/\mm^nT)=\{T/\mm^jT\}_{j=1}^{n}$ of $R/\mm^n$-module $T/\mm^nT$ is cartesian (in the sense of Definition~\ref{cartesian}).
\end{lemma}
\begin{proof}
This is~\cite{mr02} Lemma 3.7.1 (i).
\end{proof}
\begin{prop}
\label{main-cart}
Suppose $H^1_\FF(\QQ_\ell,T)$ and $H^1_\GG(\QQ_\ell,T)$ are two local on $T$ at the prime $\ell$ such that
\begin{itemize}
\item[(i)] $H^1(\QQ_\ell,T)/ H^1_\FF(\QQ_\ell,T)$ is $R$-torsion-free,
\item[(ii)] $H^1_\FF(\QQ_\ell,T/\mm^nT)/H^1_\GG(\QQ_\ell,T/\mm^nT)$ is a free $R/\mm^n$-module (where $H^1_\FF(\QQ_\ell,T/\mm^nT)$ (respectively $H^1_\GG(\QQ_\ell,T/\mm^nT)$) is the local condition on $T/\mm^nT$ propagated from the local condition $\FF$ (respectively $\GG$) on $T$) in the sense of Definition~\ref{def:propagation}.
\end{itemize}
Then the local condition $\GG$ is cartesian on the quotients $\textup{Quot}_{R/\mm^n}(T/\mm^nT)=\{T/\mm^jT\}_{j=1}^{n}$ (in the sense of Definition \ref{cartesian}) of the $R/\mm^n$-module $T/\mm^nT$.
\end{prop}

\begin{proof}
Let the $R$-module $Q$ be defined by the exactness of the following sequence:
\begin{equation}
\label{obv-exact}
0\lra H^1_\GG(\QQ_\ell,T) \lra H^1_\FF(\QQ_\ell,T) \lra Q \lra 0 .
\end{equation}
The propagated local condition $H^1_\FF(\QQ_\ell,T/\mm^jT)$ is defined as the image of $H^1_\FF(\QQ_\ell,T)$ under the canonical homomorphism $$H^1(\QQ_\ell,T)/\mm^jH^1(\QQ_\ell,T)\hookrightarrow H^1(\QQ_\ell,T/\mm^j)$$ (which is induced from the long exact sequence for the $G_{\QQ_{\ell}}$-cohomology of the exact sequence $$0\lra T \stackrel{\pi^j}{\lra} T \lra T/\mm^j \lra 0$$ where we recall that $\pi$ is a uniformizer of $R$). In other words 
\begin{equation}
\label{eqn:propagate}
H^1_\FF(\QQ_\ell,T/\mm^jT)=\textup{im}\left\{H^1_\FF(\QQ_\ell,T) \ra \frac{H^1(\QQ_\ell,T)}{\mm^jH^1(\QQ_\ell,T)} \hookrightarrow H^1(\QQ_\ell,T/\mm^jT) \right\} 
\end{equation} 
The kernel of the first map in (\ref{eqn:propagate}) is $H^1_\FF(\QQ_\ell,T) \cap \pi^j H^1(\QQ_\ell,T)$ which is equal to $\pi^jH^1_\FF(\QQ_\ell,T)$ since we assumed that $H^1(\QQ_\ell,T)/H^1_\FF(\QQ_\ell,T)$ is $R$-torsion-free. Thus 
\begin{align*}
H^1_\FF(\QQ_\ell&,T/\mm^jT)\\&=\textup{im}\left\{\frac{H^1_\FF(\QQ_\ell,T)}{\mm^jH^1_\FF(\QQ_\ell,T)} \hookrightarrow \frac{H^1(\QQ_\ell,T)}{\mm^jH^1(\QQ_\ell,T)} \hookrightarrow H^1(\QQ_\ell,T/\mm^jT) \right\} 
\end{align*}

Similarly, 
\begin{equation}
\label{eqn:propagateG}
H^1_\GG(\QQ_\ell,T/\mm^jT)=\textup{im}\left\{H^1_\GG(\QQ_\ell,T) \ra \frac{H^1(\QQ_\ell,T)}{\mm^jH^1(\QQ_\ell,T)} \hookrightarrow H^1(\QQ_\ell,T/\mm^jT) \right\} 
\end{equation}
 and the kernel of the first map in (\ref{eqn:propagateG}) is $H^1_\GG(\QQ_\ell,T) \cap \pi^j H^1(\QQ_\ell,T)$ which equals $H^1_\GG(\QQ_\ell,T) \cap \pi^j H^1_\FF(\QQ_\ell,T)$ because $H^1(\QQ_\ell,T)/H^1_\FF(\QQ_\ell,T)$ is $R$-torsion-free. We thus have 
\begin{align*}
&H^1_\GG(\QQ_\ell,T/\mm^jT)=\\&\textup{im}\left\{\frac{H^1_\GG(\QQ_\ell,T)}{\mm^jH^1_\FF(\QQ_\ell,T) \cap H^1_\GG(\QQ_\ell,T)} \hookrightarrow \frac{H^1(\QQ_\ell,T)}{\mm^jH^1(\QQ_\ell,T)} \hookrightarrow H^1(\QQ_\ell,T/\mm^jT) \right\} 
\end{align*}

Using the description which we give above of the propagated local conditions $H^1_\FF(\QQ_\ell,T/\mm^jT)$ and $H^1_\GG(\QQ_\ell,T/\mm^jT)$ we obtain an exact sequence
\begin{equation}
\label{key-exact}
0\lra H^1_\GG(\QQ_\ell,T/\mm^jT) \lra H^1_\FF(\QQ_\ell,T/\mm^jT) \lra Q/\mm^jQ.
\end{equation}
We note that the exact sequence (\ref{key-exact}) is essentially obtained by tensoring the exact sequence (\ref{obv-exact}) by $R/\mm^j$ and completing the tensored sequence to an exact sequence on the left.

To prove the statement of the Proposition we need to prove that \begin{equation}
\label{def-eqn:cartesian} H^1_\GG(\QQ_\ell,T/\mm^iT)=\ker\left\{ H^1(\QQ_\ell,T/\mm^iT) \stackrel{[\pi^{j-i}]}{\lra} \frac{H^1(\QQ_\ell,T/\mm^jT)}{H^1_\GG(\QQ_\ell,T/\mm^jT)} \right\}\end{equation} for $0<i\leq j \leq n$,  where $[\pi^{j-i}]$ is the map induced on the cohomology groups from the map $T/\mm^iT \stackrel{\pi^{j-i}}{\lra}T/\mm^jT$.  Now if $c \in H^1(\QQ_\ell,T/\mm^iT)$ and $[\pi^{j-i}]c \in H^1_\GG(\QQ_\ell,T/\mm^jT) \subset H^1_\FF(\QQ_\ell,T/\mm^jT)$ it follows from Lemma~\ref{Lemma371} that $c \in H^1_\FF(\QQ_\ell,T/\mm^jT)$. Thus (\ref{def-eqn:cartesian}) is equivalent to the statement 
\begin{equation} \label{eqn:cartesian} H^1_\GG(\QQ_\ell,T/\mm^iT)=\ker\left\{ H^1_\FF(\QQ_\ell,T/\mm^iT) \stackrel{[\pi^{j-i}]}{\lra} \frac{H^1_\FF(\QQ_\ell,T/\mm^jT)}{H^1_\GG(\QQ_\ell,T/\mm^jT)} \right\}
\end{equation}  for $0<i\leq j \leq n$.

To see (\ref{eqn:cartesian}) holds consider the following commutative diagram (where the rows come from the exact sequence (\ref{key-exact})): 
$$
\xymatrix{0 \ar[r]& H^1_\GG(\QQ_\ell,T/\mm^iT) \ar[r] \ar[d]^{[\pi^{j-i}]}& H^1_\FF(\QQ_\ell,T/\mm^iT) \ar[r] \ar[d]^{[\pi^{j-i}]}&  Q/\mm^iQ \ar[r]\ar[d]^{\pi^{j-i}} & 0\\
0 \ar[r] & H^1_\GG(\QQ_\ell,T/\mm^jT) \ar[r] & H^1_\FF(\QQ_\ell,T/\mm^jT) \ar[r]&  Q/\mm^jQ \ar[r]& 0
}
$$

By our assumption that $H^1_\FF(\QQ_\ell,T/\mm^nT)/H^1_\GG(\QQ_\ell,T/\mm^nT) \stackrel{\sim}\ra Q/\mm^nQ$ is a free $R/\mm^n$-module it follows that the right vertical map in the diagram above is injective for $0 < i \leq j \leq n$. This shows that the map $$H^1_\FF(\QQ_\ell,T/\mm^iT)\Big{/}H^1_\GG(\QQ_\ell,T/\mm^iT) \stackrel{[\pi^{j-i}]}{\lra}H^1_\FF(\QQ_\ell,T/\mm^jT)\Big{/}H^1_\GG(\QQ_\ell,T/\mm^jT)$$ is injective for $0<i\leq j\leq n$, which proves (\ref{eqn:cartesian}) and the Proposition.
\end{proof}

\begin{cor}
\label{KS-vanishing}
Suppose $\FF$ is a Selmer structure on $T$ and hypotheses \textup{\textbf{H1}-\textbf{H6}} are satisfied for $(T,\FF,\PP)$. Suppose further that the core Selmer rank $\XX(T,\FF)$ is one. Let $\GG$ be another Selmer structure on $T$, and suppose the local condition at $\ell$ determined $\GG$ satisfies the assumptions of Theorem~\ref{main-cart} and that $H^1_\GG(\QQ_\ell,T) \subsetneq H^1_\FF(\QQ_\ell,T)$. Then $$\textup{\textbf{KS}}(T/\mm^nT,\GG,\PP_n)=0.$$  
\end{cor}
\begin{proof}
It follows from Proposition~\ref{main-cart} that the hypotheses \textup{\textbf{H1}-\textbf{H6}} are satisfied by $(T/\mm^nT,\GG,\PP_n)$. Further, since $H^1_\GG(\QQ_\ell,T) \subsetneq H^1_\FF(\QQ_\ell,T)$, it follows from \cite{wi}~Proposition 1.6, \cite{mr02}~Definition 4.1.11 of the core Selmer rank and~\cite{mr02} Proposition 4.1.4 that $$0\leq \XX(T/\mm^nT,\GG) < \XX(T/\mm^nT,\FF)=1,$$ hence $\XX(T/\mm^nT,\GG)=0$. Now Theorem~\ref{thmpa} shows that $$\textup{\textbf{KS}}(T/\mm^nT,\GG,\PP_n)=0,$$ as desired.  
\end{proof}

\section{Applications}

Until the end of this paper we assume that $R$ is the ring of integers of a finite extension of $\QQ_p$ and let $F$ be its field of fractions. Let $\FFc$ be the canonical Selmer structure as in Definition~\ref{canonical selmer}. Suppose $\FF_{\textup{u-}\ell}$ is the Selmer structure defined as follows:
\begin{itemize}
\item $\Sigma(\FF_{\textup{u-}\ell})=\Sigma(\FFc),$
\item if $q \in \Sigma(\FF_{\textup{u-}\ell})$ and $q \neq \ell$ then $H^1_{\FF_{\textup{u-}\ell}}(\QQ_{q},T)= H^1_{\FFc}(\QQ_{q},T)$, 
\item $H^1_{\FF_{\textup{u-}\ell}}(\QQ_{\ell},T)=H^1_{\textup{unr}}(\QQ_{\ell},T),$
\end{itemize}
where $H^1_{\textup{unr}}(\QQ_{\ell},T)=\ker\{H^1(\QQ_{\ell},T)\ra H^1(\QQ_{\ell}^{\textup{unr}},T)\}$ is the unramified cohomology.

\begin{rem}
\label{rem:ur-finite comparison}
By~\cite{r00} Lemma I.3.5 $H^1_{\FF_{\textup{u-}\ell}}(\QQ_{\ell},T)\subset H^1_{\FFc}(\QQ_{\ell},T)$ and $$H^1_{\FFc}(\QQ_{\ell},T)/H^1_{\FF_{\textup{u-}\ell}}(\QQ_{\ell},T) \stackrel{\sim}{\lra}H^0(\QQ_{\ell},W^{I_{\ell}}\Big{/}V^{I_{\ell}}/T^{I_{\ell}}),$$where $I_{\ell}\subset G_{\QQ_{\ell}}$ is the inertia subgroup, $V=T\otimes_R F$ and $W=V/T$. Note that the $R$-module $H^0(\QQ_{\ell},W^{I_{\ell}}\Big{/}V^{I_{\ell}}/T^{I_{\ell}})$ is finite and its order is the $p$-part of the Tamagawa number at $\ell$.
\end{rem}
Recall that there is a canonical map (which we call the Euler system to Kolyvagin system map) $$\textup{\textbf{ES}}(T) \lra \overline{\textup{\textbf{KS}}}(T,\FFc,\PP)$$ from the module of Euler systems to the \emph{generalized} module of Kolyvagin systems (see~\cite{mr02} Definition 3.1.6).

\begin{thm}
\label{ESKS cokernel}
Let $\FFc$ and $\FF_{\textup{u-}\ell}$ be as above. Suppose $(T,\PP)$ satisfies the hypotheses \textup{\textbf{H1-H5}}, $\XX(T,\FFc)=1$ and $n\in\ZZ_{\geq0}$ is such that $\left(H^1_{\FFc}(\QQ_\ell,T)/H^1_{\FF_{\textup{u-}\ell}}(\QQ_{\ell},T)\right) \otimes R/\mm^n$ is a free $R/\mm^n$-module of positive rank. Then $$\textup{im}\left(\textup{\textbf{ES}}(T) \ra \textup{\textbf{KS}}(T,\FFc,\PP)\right) \subset \mm^n \textup{\textbf{KS}}(T,\FFc,\PP).$$ 
\end{thm}

\begin{proof}
We begin with the remark that $\textup{\textbf{KS}}(T,\FFc,\PP)$ is canonically isomorphic to $\overline{\textup{\textbf{KS}}}(T,\FFc,\PP)$ when the core Selmer rank $\XX(T,\FFc)$ is one. We thus allow ourselves to view the map from the module of Euler systems to the generalized module of Kolyvagin systems as a map  $\textup{\textbf{ES}}(T) \lra \textup{\textbf{KS}}(T,\FFc,\PP)$ (and this is how the statement of the Theorem makes sense). Under the assumptions above, $\textup{\textbf{KS}}(T,\FFc,\PP)$ is an $R$-module of rank one. Consider the map \begin{equation}
\label{eqn:es-ks}
\textup{\textbf{ES}}(T) \ra \textup{\textbf{KS}}(T,\FFc,\PP) \ra \textup{\textbf{KS}}(T/\mm^nT,\FFc,\PP_n).
\end{equation} 
Since $\textup{\textbf{KS}}(T,\FFc,\PP)$ is free of rank one, the statement of the Theorem is equivalent to the statement that the map~(\ref{eqn:es-ks}) is zero.
As pointed out in~\cite{mr02} Remark A.5, the proof of \cite{mr02}~Theorem 3.2.4 shows that the map (\ref{eqn:es-ks}) factors as follows:
$$
\xymatrix{
\textup{\textbf{ES}}(T) \ar[r]\ar[rd]& \textup{\textbf{KS}}(T,\FFc,\PP) \ar[r] &  \textup{\textbf{KS}}(T/\mm^nT,\FFc,\PP_n)\\
&\textup{\textbf{KS}}(T/\mm^nT,\FF_{\textup{u-}\ell},\PP_n) \ar[ur]&
}
$$
Thus it suffices to prove that $\textup{\textbf{KS}}(T/\mm^nT,\FF_{\textup{u-}\ell},\PP_n)=0$. This follows immediately from Corollary~\ref{KS-vanishing} applied with $\FF=\FFc$ and $\GG=\FF_{\textup{u-}\ell}$. Note that our assumptions guarantee that Corollary~\ref{KS-vanishing} applies with the choices above.
\end{proof}

Let $\QQ_\infty$ be the (cyclotomic) $\ZZ_p$-extension of $\QQ$, $\Gamma=\textup{Gal}(\QQ_\infty/\QQ)$ be its Galois group and $\Lambda=\ZZ_p[[\Gamma]]$ be the cyclotomic Iwasawa algebra. Let $\textup{\textbf{KS}}(T\otimes\Lambda,\FFc)$ be the module of $\Lambda$-adic Kolyvagin systems for $T$ (defined as in~\cite{kbbonline} \S3.2). Under certain hypotheses (see~\cite{kbbonline} \S2.2) it is proved that the $\Lambda$-module  $\textup{\textbf{KS}}(T\otimes\Lambda,\FFc)$ is free of rank one and that the specialization map $$\textup{\textbf{KS}}(T\otimes\Lambda,\FFc) \lra \textup{\textbf{KS}}(T,\FFc,\PP)$$ is surjective. We remark that the hypotheses \textbf{H.T} of~\cite{kbbonline} \S2.2 holds if $p$ does not divide any of the Tamagawa numbers at any prime $\ell\neq p$. 

If, however, $p$ does divide at least one Tamagawa number then the specialization map above is not surjective and it is predicted in~\cite{kbbonline} Remark 3.25 that the cokernel of this map should be related to Tamagawa numbers. As a justification of this remark one may prove:
\begin{thm}
\label{KSLambda to KS}
Suppose all the assumptions of the Theorem~\ref{ESKS cokernel} hold for $(T,\FFc,\PP)$ and $\FF_{\textup{u-}\ell}$. Let $n\in\ZZ^+$ be as in Theorem~\ref{ESKS cokernel}. Then $$\textup{im}\left( \textup{\textbf{KS}}(T\otimes\Lambda,\FFc) \lra \textup{\textbf{KS}}(T,\FFc,\PP) \right) \subset p^n \textup{\textbf{KS}}(T,\FFc,\PP).$$
\end{thm}
\begin{proof}
The proof of Theorem~\ref{ESKS cokernel} applies in an identical way, by \cite{colmez-reciprocity} Proposition II.1.1 (used instead of the proof of Theorem 3.2.4 of \cite{mr02}).
\end{proof}

We now exhibit a particular application of Theorem~~\ref{ESKS cokernel}: We apply it with Kato's Euler system for the Tate module of an elliptic curve. Let $E_{/\QQ}$ be an elliptic curve defined over $\QQ$ and let $T=T_p(E)$ be its $p$-adic Tate module. We will also assume that 
\be\label{p3} p>3,  \ee
\be\label{surj-serre} \hbox{the } p\hbox{-adic representation }G_\QQ \ra \textup{Aut}(E[p^\infty]) \hbox{ is surjective.}\ee

Suppose the Tamagawa number $c_\ell=|E(\QQ_\ell)/E_0(\QQ_{\ell})|$ at $\ell\neq p$ is divisible by $p$, and set $n=\textup{ord}_p(c_\ell)$. Since we assumed $p>3$, this shows (see~\cite{silverman} Corollary C.15.2.1) that $E$ has split multiplicative reduction at $\ell$ (thus $E_{/\QQ_\ell}$ is a Tate curve $\textup{Tate}_q$ with Tate parameter $q \in \ell\ZZ_\ell$) and that the component group of the special fiber of the N\'eron model $\mathcal{E}_{/ \textup{Spec}(\ZZ_\ell)}$ of $E_{/\QQ_\ell}$ is a cyclic group isomorphic to  $\ZZ/c_\ell\ZZ$. Thus we have an exact sequence
\begin{equation}
\label{tamagawaseq}
0\lra E_0(\QQ_\ell)\lra E(\QQ_\ell) \lra \ZZ/c_\ell\ZZ\lra 0.
\end{equation}
Further, one also has $\ZZ_\ell^{\times} \stackrel{\sim}{\lra} E_0(\QQ_\ell)$ under Tate uniformization (see for example \cite{silverman} Theorem~C.14.1). This shows that $X[p^{\infty}]$ is finite and $X[p^{k}]\cong X/p^kX$ for $X=E_0(\QQ_\ell)$ or $X=E(\QQ_\ell)$; and for $k\in\ZZ^+$. Here $X[p^{\infty}]$ stands for the $p$-power torsion inside the group $X$.

One can check that $H^1_{\textup{f}}(\QQ_\ell,V)=0$, hence the restriction map $$ H^1(\QQ_\ell,V) \lra H^1(\QQ_\ell^{\textup{unr}},V)$$ is injective, thus 
\begin{align*}  H^1_{\FFc}(\QQ_\ell,T) = H^1_{\textup{f}}(\QQ_\ell,T):=&
\ker\{H^1(\QQ_\ell,T) \ra H^1(\QQ_\ell^{\textup{unr}},V)\}\\ &=\ker\{H^1(\QQ_\ell,T) \ra H^1(\QQ_\ell,V)\},\end{align*}
which equals the image of $E(\QQ_\ell)[p^{\infty}]$ (by, for example, \cite{tate4} Proposition 2.4) inside $H^1(\QQ_\ell,T)$. See also \cite{r00} \S1.6.4. Similarly, one may show that $H^1_{\textup{unr}}(\QQ_\ell,T)$ is the image of $E_0(\QQ_\ell)[p^{\infty}]$ inside $H^1(\QQ_\ell,T)$. The diagram below summarizes our discussion in this paragraph:
$$\xymatrix{H^1_{\textup{f}}(\QQ_\ell,T) \ar @{=}[r] &\textup{im}\{E(\QQ_\ell)[p^\infty] \ar@{^{(}->}[r] & H^1(\QQ_\ell,T) \} \\
H^1_{\textup{unr}}(\QQ_\ell,T) \ar @{=}[r] \ar[u]^{\cup}&\textup{im}\{E_0(\QQ_\ell)[p^\infty] \ar@{^{(}->}[r] & H^1(\QQ_\ell,T) \}
}
$$
This, together with (\ref{tamagawaseq}), Example~\ref{core-ell} and Corollary~\ref{ESKS cokernel} (with $n=\textup{ord}_p(c_\ell)$) shows 
\begin{cor}
\label{elliptic cokernel}
Let $T=T_p(E), E, c_\ell, n$ be as above. Then

$$\textup{im}\left(\textup{\textbf{ES}}(T) \ra \textup{\textbf{KS}}(T,\FFc,\PP)\right) \subset p^n \textup{\textbf{KS}}(T,\FFc,\PP).$$  
\end{cor}

Let $N$ be the conductor of $E$. Kato~\cite{kato} has constructed an Euler system which gives rise to a Kolyvagin system $\kappa^{\textup{Kato}} \in \textup{\textbf{KS}}(T,\FFc,\PP)$ for a suitably chosen set of primes $\PP$, see~\cite{ru98} \S3.5 and \cite{mr02} \S6.2 for more details. Corollary~\ref{elliptic cokernel} shows that $\kappa^{\textup{Kato}} \in p^n \textup{\textbf{KS}}(T,\FFc,\PP).$ Further we know (Theorem~\ref{thmdvr}) that $\textup{\textbf{KS}}(T,\FFc,\PP)$ is a free $\ZZ_p$-module of rank one, and is generated by a primitive Kolyvagin system. We fix such a generator $\kappa^E$ so that $\kappa^{\textup{Kato}}=p^\alpha\cdot\kappa^E$ for some $\alpha \geq n$.

The following Theorem is the main application of $\kappa^{\textup{Kato}}$. Let $L(E,s)$ denote the Hasse-Weil $L$-function attached to $E$, and $L_N(E,s)$ the \emph{non-primitive} $L$-function which is obtained  by removing the Euler factors at primes dividing the conductor $N$ of $E$. Let $\Omega_E$ be the fundamental period of $E$, and $\textup{TS}_E$ be the Tate-Shafarevich group of $E$.
\begin{thm}
\label{Kato-main}\textup{(Kato \cite{kato})}
Assume (\ref{p3}) and (\ref{surj-serre}) holds. Suppose further that
\begin{itemize}
\item $E$ has good reduction at $p$,
\item $p \nmid $ $E(\mathbb{F}_p)$,
\item $p$ does not divide the integer $r_E$ of~\cite{ru98} Theorem 7.1,
\item $L(E,1) \neq 0$.
\end{itemize}
  Then $$\textup{length}(\textup{TS}_E[p^\infty]) \leq \textup{ord}_p(L_N(E,1)/\Omega_E).$$
\end{thm}

To prove this Theorem one utilizes Kolyvagin system machinery with $\kappa^{\textup{Kato}}$ to bound the classical Selmer group (see for example \cite{ru98} Theorem 3.2) and then use Kato's calculations (see the proof of \cite{ru98} Theorem 8.6) with $\kappa^{\textup{Kato}}_1$ (which appears as $c_\QQ$ in~\cite{ru98}). One could use $\kappa^{E}$ instead of $\kappa^{\textup{Kato}}=p^\alpha\cdot\kappa^{E}$ to bound the classical Selmer group (which, in a sense, is a "better" Kolyvagin system), and use Kato's calculations for $\kappa^{\textup{Kato}}_1$ together with the simple relation $\kappa^{\textup{Kato}}_1=p^\alpha\cdot\kappa^{E}$ to prove the folowing stronger version of Theorem~\ref{Kato-main}:
\begin{cor}
\label{selmersize}
Assume the hypotheses of Theorem~\ref{Kato-main} holds. Then
$$\textup{length}(\textup{TS}_E[p^\infty]) \leq \textup{ord}_p\left(\frac{L_N(E,1)}{p^\alpha\cdot\Omega_E}\right) \leq \textup{ord}_p\left(\frac{L_N(E,1)}{c_\ell\cdot\Omega_E}\right).$$
\end{cor}

\section{Concluding Remarks}
\label{concluding}
\subsection{More on Kato's Euler system} Explicit calculations carried out by Kato~\cite{kato} to determine a bound on the size of the Selmer group are limited to the case $L(E,1)\neq 0$. In this case the classical Selmer group $\mathcal{S}_E$ attached to $E$ is finite and one only has to deal with the very first term $\kappa^{\textup{Kato}}_1$ of the Kolyvagin system $\kappa^{\textup{Kato}} \in \textup{\textbf{KS}}(T,\FFc,\PP)$. In fact, in Corollary~\ref{selmersize} above we only use Corollary~\ref{elliptic cokernel} to conclude that $\kappa^{\textup{Kato}}_1=p^\alpha\cdot\kappa^{E}_1$, where $\alpha$ and  $\kappa^{E}$ are as above. This is sufficient in the setting of Theorem~\ref{Kato-main}. 

However one should note that Corollary~\ref{elliptic cokernel} says much more than the comparison above for the initial terms of these Kolyvagin systems, it in fact says that \be\label{div-general}\kappa^{\textup{Kato}}_r=p^\alpha\cdot\kappa^{E}_r\ee for every $r \in \NN(\PP)$, where $\NN(\PP)$ is the set of integers which are square free products of primes in $\PP$. If $\kappa^{\textup{Kato}}_1$ (hence also $\kappa^E_1$) is non-zero (which essentially put us in the setting of Theorem~\ref{Kato-main}) then  (\ref{div-general}) easily follows from the equality $\kappa^{\textup{Kato}}_1=p^\alpha\cdot\kappa^{E}_1$ and Theorem~\ref{thmpa}. Of course this is not always the case and (\ref{div-general}) is a much stronger statement in general. Unfortunately, when the Hasse-Weil $L$-function vanishes at $s=1$ (this \emph{should} amount to saying $\kappa_1^{\textup{Kato}}=0$, cf.~\cite{mr02} Corollary 5.2.13) there is no computation yet available with Kato's Kolyvagin system to exemplify the content of Corollary~\ref{elliptic cokernel} further, in terms of bounding the Selmer group (such as Corollary~\ref{selmersize}, which applies when $L(E,1)\neq0$).

\subsection{Tamagawa numbers and level lowering} 
We keep assuming $T=T_p(E)$, the $p$-adic Tate module of an elliptic curve $E_{/\QQ}$, and hypotheses (\ref{p3}) and (\ref{surj-serre}). Corollary~\ref{elliptic cokernel} says that $$\kappa^{\textup{Kato}} \in c_\ell \cdot\textup{\textbf{KS}}(T,\FFc,\PP).$$ This statement is reflected in Corollary~\ref{selmersize} as an improvement to Kato's Theorem~\ref{Kato-main}. However, this method captures only one Tamagawa factor. What if there are more then one Tamagawa numbers which are divisible by $p$? A natural question to ask is:

\textbf{Question 1.}Is it true that $\kappa^{\textup{Kato}} \in \prod_{\ell|N}c_\ell \cdot\textup{\textbf{KS}}(T,\FFc,\PP)$? 

This question turns out to be more delicate when there is more than one Tamagawa number which is divisible by $p$.  We first consider a preliminary version of Question 1. Let $d$ be the number of primes $\ell$ for which $p|c_\ell$. 

\textbf{Question 2.} Is it true that $\kappa^{\textup{Kato}} \in p^d\textup{\textbf{KS}}(T,\FFc,\PP)$? 

To address this question we consider the newform $f_E$ of level $N$ associated with $E$. Since we assumed $p>3$, $p|c_\ell$ implies that $E$ has split multiplicative reduction at $\ell$. Thus if $p|c_\ell$, then $\ell||N$. The assumption that $p|c_\ell$ in fact translates into the statement that the Galois representation $E[p]\cong T/pT$ is \emph{finite} at $\ell$. Thus, we may apply level lowering theorem of Ribet~\cite{ribet} to arrive at a modular form $g$ of level $N/\ell$ and a Galois representation $T_g$ attached to $g$ such that $T_g/pT_g \cong T/pT$. Note that $\ell$ is no longer a bad prime for $g$. The author is quite curious to see whether the Euler system for the modular form $g$ could play a role in this context to answer Question 2.

 More generally, Dummigan~\cite{dummigan} has studied the Tamagawa factors of modular forms and level lowering for their mod $p^n$ representations for $n>1$. Similarly, one might try to approach more general Question~1 via Dummigan's more general level lowering results.  

\subsection{Tamagawa numbers for higher dimensional Abelian varieties}
The discussion above with Kato's Euler system (for elliptic curves) suggests that if one would like to apply  Theorem~\ref{ESKS cokernel} with an abelian variety $A$ of higher dimension, one should understand the structure of the $p$-part $\Phi_p$ of the component group of $A/\QQ_ell$ for each $\ell\neq p$. For example, when $A=E$ is one dimensional Kodaira-N\'eron Theorem (\cite{silverman} Corollary C.15.2.1) shows that $\Phi_p$ is always cyclic if $p>2$, and this is exactly what we use to prove Corollary~\ref{elliptic cokernel}.

Let $A_{/\QQ_\ell}$ be an abelian variety (with $\ell\neq p$). Let $\mathcal{A}$ be its N\'eron model over $\ZZ_\ell$. Suppose $t$ (resp. $u$) denote the dimension of the toric (resp. unipotent)  part of the special fiber $\mathcal{A}_s$. Let $\Phi[p]$ denote the $p$-torsion of $\Phi$. Since $\Phi$ is a finite group, we have $\Phi[p]\cong \Phi_p/p\Phi_p$. Theorem below can be found in \cite{abbes}:
\begin{thm}\textup{(\cite{abbes} Proposition 5.13 (i))}
\label{component-rank} If $p\geq 3$,  
$$\textup{dim}_{\mathbb{F}_p}\Phi[p]\leq t+u\leq \textup{dim} A.$$
\end{thm}

Suppose now that $A_{/\QQ}=A_f$ is an abelian variety of attached to a newform $f$ of level $N$. Then $A_f$ is an abelian variety of $\textup{GL}_2$-type: Let $\mathcal{R}$ denote the ring generated by the Fourier coefficients of $f$ and set $K=\QQ\otimes \mathcal{R}_f$, then $\mathcal{R} \subset \textup{End} _\QQ(A)$ and $[K:\QQ]=\textup{dim} A$. We assume the following additional hypotheses on $p$:
\be
\label{p-inert}
\mathcal{R} \otimes\mathbb{F}_p \hbox{ is a field, i.e. } p \hbox{ is inert in } \mathcal{R}.
\ee 

$\mathcal{R}\otimes\mathbb{F}_p$ acts on $\Phi[p]$, thus Theorem~\ref{component-rank} and assumption~(\ref{p-inert}) shows that
\begin{cor}
\label{R-structure}
$\Phi[p]$ is a cyclic $\mathcal{R}/p\mathcal{R}$-module.
\end{cor}

Let $\mathcal{R}_p$ denote the completion of the ring $R$ at $p$. The action of $\mathcal{R}$ on $\Phi_p$ naturally extends to an action of $\mathcal{R}_p$ on $\Phi_p$. Corollary~\ref{R-structure} and Nakayama's lemma implies the following

\begin{cor}
\label{R_p structure}
Under the hypotheses above $\Phi_p$ is a cyclic $\mathcal{R}_p$-module.   
\end{cor}

Let $\mathcal{O}$ denote the maximal order of the number field $K$. One easily deduces from the assumption (\ref{p-inert}) that the ring homomorphism $\mathcal{R}/p\mathcal{R}\ra\mathcal{O}/p\mathcal{O}$ is an isomorphism, hence $p$ is inert in the extension $\mathcal{O}/\ZZ$ as well. Further, the inclusion $\mathcal{R} \hookrightarrow \mathcal{O}$ induces a natural isomorphism $\mathcal{R}_p \stackrel{\sim}{\lra} \mathcal{O}_p$. Thus we proved

\begin{prop}
\label{O_p structure}
If we assume (\ref{p-inert}) then $\Phi_p$ is a cyclic $\mathcal{O}_p$-module.   
\end{prop}
Note that this is the analogous statement for abelian varieties $A_f$ to that for elliptic curves (i.e. to the case $K=\QQ$), which, in that case, is implied by the Kodaira-N\'eron Theorem.

Consider the $p$-adic Tate module $T_p(A_f)$, and set $$V_p(A_f)=\QQ_p\otimes_{\ZZ_p}T_p(A_f).$$ We keep assuming~(\ref{p-inert}). It is known that $V_p(A_f)$ is a $K_p$-vector space of dimension 2; choose a $G_\QQ$-invariant lattice $T$ inside $V_p(A_f)$ so that $T$ is a free $\mathcal{O}_p$-module of rank 2. Kato~\cite{kato} has constructed an Euler system for this $\mathcal{O}[[G_\QQ]]$ representation. One could use Theorem~\ref{ESKS cokernel} (which applies thanks to Proposition~\ref{O_p structure} with $R=\mathcal{O}_p$) to obtain similar results to Corollary~\ref{elliptic cokernel} in this setting.
\bibliographystyle{alpha}
\bibliography{tam}
\end{document}